\documentclass[a4paper,11pt]{amsart}
\usepackage{amsmath}
\usepackage{amscd}
\usepackage{mathtools}
\usepackage{mathrsfs}
\usepackage{kotex}

\usepackage{graphicx}
\usepackage{tikz,  tikz-cd}

\usepackage[utf8]{inputenc}

\usetikzlibrary{arrows}
\numberwithin{equation}{section}
\usepackage{epsfig}
\usepackage{amsfonts,amssymb,amsopn}

\usepackage{amsthm}
\usepackage{hyperref}
\usepackage{verbatim}
\usepackage{color}
\usepackage[color,all]{xy}

\newcommand{\cM}{{\mathcal M}}
\newcommand{\cO}{{\mathcal O}}

\newcommand{\cc}{{\mathbb C}}
\newcommand{\pp}{{\mathbb P}}

\newcommand{\bA}{{\mathscr A}}

\newcommand{\Sym}{\mathrm{Sym}}
\newcommand{\End}{\mathrm{End}\,}

\newcommand{\h}{\widehat}

\newcommand{\Om}{\Omega}

\newtheorem{theorem}{{\textbf Theorem}}
\newtheorem{proposition}[theorem]{{\textbf Proposition}}

\newtheorem{lemma}[theorem]{{\textbf Lemma}}

\newtheorem{remit}[theorem]{{\textbf Remark}}
\newtheorem{ex}{{\textbf Example}}

\newenvironment{remark}{\begin{remit}\rm}{\end{remit}}
\newtheorem{defn}[theorem]{{\textbf Definition}}
\newenvironment{definition}{\begin{defn}\rm}{\end{defn}}

\title[Deformations of the tangent bundle]{Deformations of the tangent bundle  of a projective hypersurface}

\author{Insong Choe}

\address{Department of Mathematics, Konkuk University, Neungdong-ro, Gwangjin-gu, Seoul 05029, Korea}

\email{ischoe@konkuk.ac.kr}

\author{Kiryong Chung}

\address{Department of Mathematics Education, Kyungpook National University, 80 Daehakro, Bukgu, Daegu 41566, Korea}

\email{krchung@knu.ac.kr}

\author{Jun-Muk Hwang}

\address{Institute for Basic Science, Center for Complex Geometry, Daejeon 34126, Korea}

\email{jmhwang@ibs.re.kr}

\thanks{Jun-Muk Hwang was supported by the Institute for Basic Science (IBS-R032-D1). }

\begin{document}

\begin{abstract}
For a nonsingular hypersurface $X \subset \pp^n, n \geq 4,$ of degree $d \geq 2$, we show that the space $H^1(X, \End(T_X))$ of infinitesimal deformations of the tangent bundle $T_X$ has dimension ${n+d-1 \choose d} (d-1)$ and all infinitesimal deformations are unobstructed even though $H^2(X, \End(T_X))$ can be nonzero. Furthermore, we prove that the irreducible component of the moduli space of stable bundles containing  the tangent bundle is a rational variety, by constructing an explicit birational model.
\end{abstract}

\maketitle



\section{Introduction}\label{s.intro}
We work over complex numbers. For a vector bundle $E$ on a variety $X$, we use the notation $h^i(X, E) := \dim H^i(X, E).$

 Let $X \subset \pp^n, n \geq 4,$ be a nonsingular  hypersurface  of  degree $d\geq 2.$ Let $T_X$ be its tangent bundle and denote by $\End(T_X)$ (resp. $\End_0(T_X)$)  the bundle of endomorphisms (resp. traceless endomorphisms) of $T_X$. Then $H^1(X, \End(T_X)) = H^1(X, \End_0(T_X))$ (from $H^1(X, \cO_X) =0$) describes the infinitesimal deformations of $T_X$ as vector bundles.  In \cite[Theorem 6.1]{P},  Peternell proved that $H^1(X, \End_0(T_X)) \neq 0$, namely, the tangent bundle has nontrivial infinitesimal deformations. This naturally leads to the following questions. \begin{itemize}
\item[{\bf Q1}] What is the exact value of $h^1(X, \End (T_X))$?
\item[{\bf Q2}] Are all elements of $H^1(X, \End (T_X))$ unobstructed? Equivalently, is the deformation space of $T_X$ smooth at the point corresponding to $T_X$?
\end{itemize}
Furthermore, the tangent bundle $T_X$ is a stable bundle with respect to the polarization $\mathcal{O}(1)$ on $\pp^n$ by \cite[Corollary 0.3]{PW}. Thus $T_X$ represents a point $[T_X]$ in  the
moduli scheme of stable bundles on $X$. Let $\cM$ be the irreducible component of the moduli scheme containing $[T_X]$.
\begin{itemize}
\item[{\bf Q3}] What is the geometric nature of the algebraic variety  $\cM$?
\end{itemize}
 Our goal is to study these three questions.
 The answer to {\bf Q1} is the following, to be proved in Section \ref{s.dim}.
 
\begin{theorem}\label{t.dim}
Let $X \subset \pp^n, n \geq 4,$ be a nonsingular projective hypersurface of degree $d \geq 2.$ Then
$h^1(X,  {\rm End}_0(T_X) )  = {n+d-1 \choose d} (d-1) $.
\end{theorem}

The obstruction to realizing an infinitesimal deformation in $H^1(X,  {\rm End}(T_X) )$ as an actual deformation of $T_X$ lies in $H^2(X,  {\rm End}(T_X) ) = H^2(X, T_X \otimes \Omega_X).$ Thus the answer to {\bf Q2} would have been straightforward if $H^2(X, T_X \otimes \Omega_X) =0.$  However, this group does not always vanish as  the following theorem says, which is proved in Section \ref{s.H2}.

\begin{theorem}\label{t.H2}
Let $X \subset \pp^n, n \geq 4,$ be a nonsingular projective hypersurface of degree $d \geq 2.$ \begin{itemize} \item[(i)] $H^2(X, T_X \otimes \Omega_X) =0$ if $n \geq 6$.
\item[(ii)] When $n=4$, 
\[
h^2( X, T_X \otimes \Omega_X) =
\begin{cases}
0,  &\text{if} \ d =2, 3\\
45,  &\text{if} \ d=4\\
224,  & \text{if} \ d=5.
\end{cases}
\] \end{itemize} \end{theorem}

Consequently, {\bf Q2} is a subtler question for some values of $n$ and $d$. 
Nevertheless,  we answer {\bf Q2} affirmatively in Section \ref{s.Main} as follows for all $n \geq 4$ and $d \geq 2$  by an explicit construction of deformations of $T_X$, which also gives a partial answer to {\bf Q3}.

\begin{theorem}\label{t.Main}
Let $X \subset \pp^n, n \geq 4,$ be a nonsingular projective hypersurface of degree $d \geq 2.$ Then there is  an affine space $\bA$   with a base point $o \in \bA$ and a  torsionfree sheaf  $\mathbb{E}  \to X \times \bA$  which satisfy the following properties.
\begin{itemize}
\item[(i)]  $\dim \bA = {n+d-1 \choose d} (d-1)  $.
\item[(ii)]  The restriction $\mathbb{E}|_{X \times \{o\}} $  is isomorphic to the tangent bundle $T_X$.
 \item[(iii)]  There is a  Zarisky open subset $\bA^\circ \subset \bA$ with $o \in \bA^\circ$ such that for any $\alpha \in \bA^\circ$,
 the restriction $\mathbb{E}|_{X \times \{\alpha\}}$ is a vector bundle of determinant $\wedge^{n-1} ( \mathbb{E}|_{X \times \{\alpha\}}) \cong K^{-1}_X$.
\item[(iv)] There is a Zariski open subset $\bA^{\rm faithful} \subset \bA^\circ$ with $o \in \bA^{\rm faithful}$ such that for any $\alpha,\alpha' \in \bA^{\rm faithful}$, the two vector bundles $\mathbb{E}|_{X \times \{\alpha\}}$ and $\mathbb{E}|_{X \times \{\alpha'\}} $ are isomorphic only when $\alpha = \alpha'$.
 \item[(v)] There is a Zariski open subset $\bA^{\rm stable} \subset \bA^\circ$ with $o \in \bA^{\rm stable}$  such that for any $\alpha \in \bA^{\rm stable}$,  the restriction $\mathbb{E}|_{X \times \{\alpha\}} $ is a stable vector bundle.
\end{itemize}
Consequently, the moduli space $\cM$ is nonsingular at $[T_X]$ and it is a rational variety.
\end{theorem}

What happens when $n=3$ and $X$ is a surface?
A construction analogous to $\bA$ and $\mathbb{E}$ of Theorem \ref{t.Main} works even when $n=3$. However, the statement of Theorem \ref{t.dim}  does not hold when $n=3$. In fact, by a general theory, the moduli space $\cM$ at the point $[T_X]$ has dimension at least $h^1(X, T_X\otimes \Omega_X)-h^2(X, T_X\otimes \Omega_X)$. One can check (see Proposition \ref{p.n=3})  that when $n=3$ and $d \geq 3$,  $$h^1(X, T_X\otimes \Omega_X)-h^2(X, T_X\otimes \Omega_X)=\frac{1}{3}(d-1)(7d^2-5d-3),$$ which is strictly bigger than $\dim \bA=\frac{1}{2}(d-1)(d^2+3d+2)$. Thus one needs new ideas to study {\bf Q1}, {\bf Q2} and {\bf Q3} for surfaces in $\pp^3$.

We would like to comment on the cases missing in Theorem \ref{t.H2}.  When $n=4$,  we expect for all $d \geq 4,$\begin{equation}\label{eq:ob4d}
h^2(X, T_X\otimes \Omega_X) =(11d+1) {d-1 \choose 3}.
\end{equation}  
This was confirmed for the Fermat  hypersurfaces for $4 \le d \le 25$ by using the Macaulay2 program(\cite{GS}).
Also this agrees with Theorem \ref{t.H2} (ii) for $d= 4,5$. We have not been able to verify this for $d \geq 6$.  When $n=5$,  similarly we checked for Fermat hyperfaces of degree $d$ that $h^2(X, T_X \otimes  \Omega_X) =1$ for $d=3$ and $h^2(X, T_X \otimes  \Omega_X) =0$ for $4 \le d \le 25$.


To close the introduction, let us mention that
 the third-named author erroneously asserted in \cite[page 77]{Hw} that the tangent bundle of an irreducible Hermitian symmetric space cannot be deformed. This contradicts Theorem \ref{t.Main} for the quadric hypersurface, which is an irreducible Hermitian symmetric space of type IV.  The flaw was in the limiting argument in the proof of \cite[Proposition 2.1]{Hw}. Consequently, the main result Theorem 1.2 of \cite{Hw} is incorrect. It is  expected  that the deformation space of the tangent bundle of an irreducible Hermitian symmetric space, other than projective space, has positive dimension (see \cite[Question 1.6]{P}). It would be interesting to find an explicit birational model of the deformation space, generalizing Theorem \ref{t.Main}.

\section{Dimension of $H^1(X, T_X \otimes \Omega_X)$}\label{s.dim}
In this section,   we present a number of  cohomological computations associated to the tangent bundle $T_X$ of a smooth projective hypersurface $X \subset \pp^n$, including a proof of Theorem \ref{t.dim}.

First we recall the following result, Bott's formula (for example, \cite[page 8]{OSS}).

\begin{proposition}\label{p.Bott}
\[
h^i( \pp^n,  \Om^j_{\pp^n} (k)) =
\begin{cases}
{n+k-j \choose k} {k-1 \choose j},  &\text{if} \ i=0,  0 \le j \le n,  k >j\\
{j-k \choose -k} {-k-1 \choose n-j},  &\text{if} \ i=n,  0 \le j \le n,  k< j-n\\
1,  & \text{if} \ k=0,  i=j\\
0,  &\text{otherwise}.
\end{cases}
\]
\end{proposition}

From this, we can show:

\begin{lemma} \label{endT}
For $0 \le i \le n-2$ and $d \ge 2$, we have $\dim H^i(\pp^n,   \End T_{\pp^n}(-d)) =0$.
\end{lemma}
\begin{proof}
This follows from Proposition \ref{p.Bott} combined with the Euler sequence tensored by $\Om_{\pp^n}(-d)$:
\[
0 \to \Om_{\pp^n}(-d) \to  \Om_{\pp^n}(1-d)^{\oplus (n+1)} \to \End T_{\pp^n}(-d) \to 0.
\]
\end{proof}

Recall the following from \cite[Satz 8.11]{F}.

\begin{proposition}\label{special}
For a nonsingular projective hypersurface $X \subset \pp^n$ and each  $0 < i < n-1$,  we have
\[
\begin{cases}
H^i(X, \Om_X^j(k) )=0,  &\text{if} \ i+j \neq n-1, k \neq 0\\
H^i(X, \Om_X^j) = 0, & \text{if} \ i+j \neq n-1,  i \neq j \\
H^i(X, \Om_X^i) = \cc,  & \text{if} \ 2i \neq n-1.
\end{cases}
\]
\end{proposition}

\begin{lemma} \label{basic}
For a nonsingular projective  hypersurface $X \subset \pp^n, n \geq 4$ of degree $d \ge 2$, the following holds.
\begin{itemize}
\item[(i)] $h^0(X,  \Om_X(d) ) = {n+d-1 \choose d} (d-1)$.  
\item[(ii)] $h^i (X,  \Om_X(k)) =0$ for $0 < i < n-2$ and $k \neq 0$.
\item[(iii)] $h^i(X, T_{\pp^n}(-d)|_X)  =0$ for $i=0,1,2$.
\item[(iv)]  $h^1(X,  T_X(-d)) = 1$.
\end{itemize}
\end{lemma}

 \begin{proof} \
 From the vanishing $h^1(X,  \cO_X)  =0$ and the cotangent sequence
\[
0 \to \cO_X \to \Om_{\pp^n} |_X (d) \to \Om_X(d) \to 0,
\]
we have
\[
h^0(X, \Om_X(d)) = h^0(X,  \Om_{\pp^n} |_X (d)) -1.
\]
Combining Proposition \ref{p.Bott} and  the following sequence
\begin{equation}
0 \to \Om_{\pp^n} \to \Om_{\pp^n}(d) \to \Om_{\pp^n}|_X(d) \to 0,
\end{equation}
we obtain
\[
h^0(X, \Om_X(d)) =  h^0(X, \Om_{\pp^n} |_X (d)) -1  = h^0(\pp^n,  \Om_{\pp^n}(d)) = {n+d-1 \choose d} (d-1),
\] which proves (i).

 (ii)  is immediate from Proposition  \ref{special}.

  Since
$T_{\pp^n} \cong \Om_{\pp^n}^{n-1} (n+1)$,    for $i=0,1,2$,  we have
\[
h^i(\pp^n,  T_{\pp^n}(-d)) = h^i (\pp^n ,   \Om_{\pp^n}^{n-1} (n+1-d) ) =0
\]
 from Proposition \ref{p.Bott}.
  Similarly,   we have $h^i(\pp^n,  T_{\pp^n}(-2d)) = 0$ for $i=1,2$.  Furthermore, Proposition \ref{p.Bott} gives
  \[
  h^3(\pp^n,  T_{\pp^n}(-2d)) =  h^3 (\pp^n ,   \Om_{\pp^n}^{n-1} (n+1-2d)  )=0.
  \]
 Thus (iii) follows from the sequence
\[
0 \to T_{\pp^n} (-2d) \to T_{\pp^n}(-d) \to  T_{\pp^n}(-d )|_X \to 0.
\]

Finally, to prove (iv), note that
\[
h^{n-2} (X,  \Om_{\pp^n}(2d-n-1)|_X) =0 = h^{n-1} (X, \Om_{\pp^n}(2d-n-1)|_X), 
\]
 which follows from Proposition \ref{p.Bott} combined with the sequence
$$ 0 \to \Om_{\pp^n}(d-n-1) \to \Om_{\pp^n}(2d-n-1) \to \Om_{\pp^n}(2d-n-1)|_X \to 0.$$
Then Serre duality and $K_X \cong \cO_X(d-n-1)$ give
   $$h^{n-1} (X, \cO_X (d-n-1)) = h^0(X, \cO_X (n+1-d) \otimes K_X) = h^0(X, \cO_X) =1$$ and
$$h^1(X,  T_X(-d)) =h^{n-2} (X,  \Om_X(2d-n-1)) = h^{n-1} (X, \cO_X (d-n-1)) =1.$$
\end{proof}

\begin{remark} Note that $h^{n-2} (X, \Omega_X (2d-n-1)) =1$ is in contrast to   $h^{n-2}(X, \Om_X(k))=0$ for $k > 2d-n-1$ from \cite[Lemma 3.8]{L}. \end{remark}

\begin{lemma}   \label{simple}
Let $X \subset \pp^n$ be as in Lemma \ref{basic}.  Then
\begin{itemize}
\item[(i)]  $h^0(X,  \End (T_{\pp^n})|_X)  = h^0(X, T_{\pp^n}|_X \otimes \Om_X) =1 $.
\item[(ii)] $h^1(X,  \End (T_{\pp^n})|_X)  = h^1(X,     T_{\pp^n}|_X \otimes \Om_X) =0$.
\end{itemize}
\end{lemma}
\begin{proof}
Taking the tensor product of the cotangent sequence  and  $T_{\pp^n}|_X$, we obtain
\[
0 \to T_{\pp^n}(-d)|_X \to \End (T_{\pp^n})|_X \to T_{\pp^n}|_X \otimes \Om_X \to 0.
\]
By Lemma \ref{basic} (iii),  we obtain $ H^i(X,  \End (T_{\pp^n})|_X) = H^i (X, T_{\pp^n}|_X \otimes \Om_X ) $ for $i=0,1$.

From the sequence
\[
0 \to  \End (T_{\pp^n})(-d)  \to  \End (T_{\pp^n}) \to  \End (T_{\pp^n})|_X \to 0,
\] and  $H^i(\pp^n,  \End (T_{\pp^n}) (-d)) =0$ for $i=0, 1, 2$ in Lemma \ref{endT},
we have
$h^i(X, \End (T_{\pp^n})|_X) = h^i(\pp^n, \End (T_{\pp^n}))$
for $i=0,1.$    Thus Lemma \ref{simple} follows from the fact that $T_{\pp^n}$ is simple (\cite[Ch. I, Lemma 4.1.2]{OSS}) and it has no infinitesimal deformation (because the flag manifold $\pp (T_{\pp^n})$ is infinitesimally rigid).
\end{proof}

Now we are ready to prove Theorem \ref{t.dim}.

\begin{proof}[Proof of Theorem \ref{t.dim}]  Consider the long exact sequence
\begin{eqnarray*} \lefteqn{
0 \to H^0(X, T_X \otimes \Om_X) \to H^0(X, T_{\pp^n}|_X \otimes \Om_X) \to H^0(X, \Om_X(d)) } \\ & &  \to H^1(X, T_X \otimes \Om_X) \to H^1(X, T_{\pp^n}|_X \otimes \Om_X) \to \cdots.
\end{eqnarray*}
Recall that $T_X$ is stable from \cite[Corollary 0.3]{PW}, hence $h^0(X, T_X \otimes \Om_X) = 1.$ Thus
 Lemma \ref{simple} implies that $H^0(X,  \Om_X(d))  \cong H^1(X, T_X \otimes \Om_X) $.
Then Lemma \ref{basic} (i) gives the value of $h^1(X, T_X \otimes \Om_X) = h^1(X, \End_0(T_x)).$
\end{proof}

\section{Dimension of $H^2(X, T_X \otimes \Omega_X)$}\label{s.H2}
To prove Theorem \ref{t.H2}, recall the following.

\begin{lemma}\label{l.Chern}
Let $X$ be a smooth degree $d$ hypersurface in $\pp^n$. Let $h=c_1(\cO_X(1))$ be the hyperplane class. Then the total Chern class of $T_X$ is $(1 + h)^{n+1} \cdot (1 + dh)^{-1}.$ \end{lemma}

\begin{proof} From Euler sequence and tangent bundle sequence of $X$ in $\pp^n$, total Chern class of $T_X$ is
\begin{eqnarray*}
c(T_X) &=&\sum_{i} c_i(T_X) \ = \ c(T_{\pp^n}|_X)\cdot c(\cO_X(d))^{-1}\\
&=& c(\cO_X(1))^{n+1}\cdot c(\cO_X(d))^{-1}\ = \ \frac{(1+h)^{n+1}}{(1+dh)}.
\end{eqnarray*} \end{proof}

\begin{proof}[Proof of Theorem \ref{t.H2}]
By twisting the cotangent bundle $\Omega_X$ of $X$ into the tangent bundle sequence of $X$ in $\pp^n$ and the Euler sequence of $\pp^n$, we have
\begin{equation}\label{eq:twte}
\xymatrix{ && 0 \\ 0\ar[r]&T_X\otimes \Omega_X\ar[r]&T_{\pp^n}|_X\otimes \Omega_X\ar[r] \ar[u] & \Omega_X(d)\ar[r]&0\\
&&\Omega_X(1)^{n+1}\ar[u]&&\\
&&\Omega_X\ar[u]&& \\
&& 0 \ar[u]}
\end{equation}
We have the long exact sequences resulting from \eqref{eq:twte}:
\begin{equation}\label{longob}
\xymatrix{&&&H^{3}(\Omega_X)&\\
\cdots \ar[r] & H^{1}(\Omega_X(d))\ar[r] & H^{2}(T_X\otimes \Omega_X)\ar[r]&H^{2}(T_{\pp^n}|_X\otimes\Omega_X)\ar[u] \ar[r] & \cdots \\
&&&H^{2}(\Omega_X(1)^{n+1})\ar[u]}
\end{equation}
When $n\geq 6$,  Proposition \ref{special} (ii) gives $H^3(\Omega_X)=0$ and Lemma \ref{basic} (ii) gives $h^1(\Omega_X(d))=h^2(\Omega_X(1))=0$.  This shows $H^2(X, T_X \otimes \Omega_X) =0$.

When $n=4$, from Lemma \ref{l.Chern},we have
\[
c(T_X)=1+(-d+5)h+(d^2-5d+10)h^2+(-d^3+5d^2-10d+10)h^3.
\]
Hence
\[
\begin{split}
ch(T_X)&=3+(-d+5)h+(-\frac{1}{2}d^2+\frac{5}{2})h^2+(-\frac{1}{6}d^3+\frac{5}{6})h^3,\\
ch(\Omega_X)&=3+(d-5)h+(-\frac{1}{2}d^2+\frac{5}{2})h^2+(\frac{1}{6}d^3-\frac{5}{6})h^3 \text{ and}\\
td(X)&=1+(-\frac{1}{2}d+\frac{5}{2})h+(\frac{1}{6}d^2-\frac{5}{4}d+\frac{35}{12})h^2-\frac{1}{24}(d-5)(d^2-5d+10)h^3.
\end{split}
\]
Therefore by Hirzebruch-Riemann-Roch formula,
\begin{equation}\label{eq:HRR}
\chi (X, T_X\otimes \Omega_X)= \int_{[X]} ch(T_X)\cdot ch(\Omega_X) \cdot td(X) = \frac{d(d-5)(13d^2-25d+10)}{8}.
\end{equation}

   \begin{itemize}
   \item If $d=2$,  then $\chi (T_X\otimes \Omega_X) =-9$ by \eqref{eq:HRR}. But $h^0(T_X\otimes \Omega_X)=1$ and  $h^3 (T_X,T_X) = h^0(T_X\otimes \Omega_X(-3)) =0$ by Serre duality and the stability of $T_X$. On the other hand, by Theorem \ref{t.dim}, we have $h^1(T_X\otimes \Omega_X)= {4+2-1 \choose 2}=10$. This shows that $h^2(T_X\otimes \Omega_X)= 0$.
   \item If $d=3$,  essentially the same argument as in $d=2$  shows $h^2(T_X\otimes \Omega_X)= 0.$ 
   \item If $d=4$,  we have $\chi (T_X\otimes \Omega_X) = -59$ by \eqref{eq:HRR}.
Since $K_X \cong \cO_X(-1)$ and $T_X$ is stable,  we have $h^3 (T_X\otimes \Omega_X) = h^0 (T_X\otimes \Omega_X \otimes K_X)  =0$.   On the other hand, we know that $h^1 (T_X\otimes \Omega_X) = 105$ by Theorem \ref{t.dim}.  Therefore $h^2(T_X\otimes \Omega_X) = 45$.
\item Let $d=5$. Since $K_X \cong \cO_X$, we have $$h^2 (T_X\otimes \Omega_X) = h^1 (T_X\otimes \Omega_X)=(5-1)\times {4+5-1 \choose 5}=224$$ by Serre duality and Theorem \ref{t.dim}.
   \end{itemize} \end{proof}

We would like to finish the section with the following result mentioned in Section \ref{s.intro}.

\begin{proposition}\label{p.n=3}
When $X \subset \pp^3$ is a surface of degree $d \geq 3$, $$ h^1(X, T_X\otimes \Omega_X)-h^2(X, T_X\otimes \Omega_X) =\frac{1}{3}(d-1)(7d^2-5d-3).$$ \end{proposition}

\begin{proof} 
From Lemma \ref{l.Chern}, the total Chern class of $T_X$ is
\[c(T_X)=1+(-d+4)h+(d^2-4d+6)h^2.
\]
Hence
\[
\begin{split}ch(T_X)&=2+(-d+4)h+(-\frac{1}{2}d^2+2)h^2,\\
ch(\Omega_X)&=2+(d-4)h+(-\frac{1}{2}d^2+2)h^2 \text{ and }\\
td(X)&=1+(-\frac{1}{2}d+2)h+(\frac{1}{6}d^2-d+\frac{11}{6})h^2.
\end{split}
\]
Therefore by Hirzebruch--Riemann--Roch formula, we have
\begin{eqnarray*}
\chi (X, T_X\otimes \Omega_X) & = & \int_{[X]} ch(T_X)\cdot ch(\Omega_X) \cdot td(X)\\
&= &-\frac{7d^3-12d^2+2d}{3}. \end{eqnarray*}
Since $H^0(X, T_X \otimes \Omega_X) =\cc$ when $d \geq 3$, we obtain $$h^1(X, T_X\otimes \Omega_X)-h^2(X, T_X\otimes \Omega_X) =1-\chi(T_X\otimes \Omega_X)=\frac{1}{3}(d-1)(7d^2-5d-3).$$ \end{proof}

\section{A birational model of $\cM$}\label{s.Main}

\begin{definition}\label{d.A} Let $V$ be a complex vector space of dimension $n+1 \geq 5$. Fix an integer $d \geq 2$.
Regard  $\Sym^{d-1} V^* \otimes V^*$ as a subspace of $\bigotimes^d V^*$ consisting of elements $\alpha \in \bigotimes^d V^*$ satisfying 
\[
\alpha (v_1, \ldots, v_{d-1}, v_d) = \alpha (v_{\sigma(1)}, \ldots, v_{\sigma(d-1)}, v_d) 
\]
 for any permutation $\sigma $ of  $\{1, \ldots, d-1\}$.
 Let $\bA \subset \Sym^{d-1}V^* \otimes V^*$ be the kernel of the symmetrizing map ${\rm sym}^d: \bigotimes^d V^* \to \Sym^d V^*$. Equivalently, it is the
 subspace consisting of  $\alpha \in \Sym^{d-1}V^* \otimes V^*$ satisfying $$\alpha  ( v, v, \ldots, v, u) = - (d-1)  \: \alpha  (u, v,\ldots, v,v) \mbox{ for any } u,v \in V.$$\end{definition}

\begin{lemma} \label{l.dim}
$\dim \bA  = {n+d-1 \choose d} \cdot (d-1)$ and $\bA \cap \Sym^d V^* = \{0\}.$
\end{lemma}

\begin{proof}
 Since the restriction  ${\rm sym}^d|_{\Sym^{d-1} V^* \otimes V^*}$ is surjective, \begin{eqnarray*} \dim \bA & = &
 \dim (\Sym^{d-1}V^* \otimes V^* ) - \dim (\Sym^d V^*) \\ &=&  {n+d-1 \choose d-1} \cdot (n+1) -  {n+d \choose d} = {n+d-1 \choose d} \cdot (d-1). \end{eqnarray*} Since ${\rm sym}^d$ sends $\Sym^d V^*  (\subset \bigotimes^d V^*) $ isomorphically to $\Sym^d V^*$, we have $\bA \cap \Sym^d V^* = \{0\}.$ \end{proof}

\begin{definition}\label{d.phi}
In Definition \ref{d.A}, let $X \subset \pp V$ be a nonsingular hypersurface cut out by a symmetric $d$-form $q \in \Sym^d V^*.$ For each $x \in X$, let $\h{x} \subset V$ be the 1-dimensional subspace determined by $x \in \pp V$. Its dual $\h{x}^*$ can be identified with the fiber of the line bundle $\cO_X(1)$ at $x$. \begin{itemize}
\item[(i)] For each $x \in X$ and each $\alpha \in \bA$, define the homomorphism $\phi^{\alpha}_x : V \otimes \h{x}^* \to \Sym^d \h{x}^*$ by $$\phi^{\alpha}_x( v \otimes \lambda)(u, \cdots, u) := \lambda(u) \cdot (q + \alpha) (u, \cdots, u, v) $$ for any $v \in V, \lambda \in \h{x}^*$ and $u \in \h{x}$.
 \item[(ii)] For each $\alpha \in \bA$, the family of homomorphisms $\{\phi^{\alpha}_x \mid x \in X\}$ from (i) determines a  homomorphism $\phi^{\alpha}: V \otimes \cO_X(1) \to \cO_X(d)$ of vector bundles on $X$.
     \item[(iii)]   In terms of the projection $\pi: X \times \bA \to X$,  (ii) determines a homomorphism $\Phi: V \otimes \pi^* \cO_X(1) \to \pi^* \cO_X(d)$ of vector bundles on $X \times \bA$.
         \item[(iv)] Since the one-dimensional subspace $\h{x} \otimes \h{x}^* \subset V \otimes \cO_x(1)$ is contained in ${\rm Ker}(\phi^{\alpha}_x)$ for each $x \in X$, we have a natural inclusion $\pi^*\cO_X \subset {\rm Ker}(\Phi).$ Define $\mathbb{E}$ as the quotient sheaf  ${\rm Ker}( \Phi) / \pi^* \cO_X$ on $X \times \bA$, to have
     the following diagram of sheaves over $X \times \bA$:
 \begin{equation*} 
\begin{tikzcd}
&  & \pi^* \cO_X(d) \arrow[r, phantom, sloped, "="]  & \pi^* \cO_X(d)  &\\
0 \arrow[r] &
 \pi^*\cO_{X} \arrow[r]  &
V \otimes \pi^*\cO_{X}(1) \arrow[r] \arrow[u,  "\Phi"] &
\pi^* (T_{\pp V}|_X)  \arrow[r]  \arrow[u]   & 0
\\
0 \arrow[r] &
\pi^* \cO_X \arrow[r] \arrow[u, phantom, sloped, "="]&
{\rm Ker}( \Phi )  \arrow[r] \arrow[u, phantom, sloped, "\subset"] &
\mathbb{E}   \arrow[r] \arrow[u, phantom, sloped, "\subset"] &
0.
\end{tikzcd}
\end{equation*}
 \end{itemize} \end{definition}

 \begin{lemma}\label{l.circ}
 In Definition \ref{d.phi}, let $\bA^\circ \subset \bA$ be the subset consisting of $\alpha \in \bA$ such that  ${\rm Ker}(\phi^{\alpha})$ is a vector bundle of rank $n-1$ on $X$.  For $\alpha \in \bA^\circ$, denote by $E^{\alpha}$ the vector bundle on $X$ corresponding to $\mathbb{E}|_{X \times \{ \alpha\}}$.   Then   $\bA^\circ$ is a Zariski open subset containing $0 \in \bA$ and  $E^0$ is isomorphic to $ T_X$.
 \end{lemma}

 \begin{proof} Since $q$ defines the hypersurface $X$, for any $x \in X$ and  $0 \neq u \in \h{x},$
the affine tangent space $T_u \h{X} \subset V$ at $u$ to the affine cone $\h{X} \subset V$ of $X \subset \pp V$ is $$T_u \h{X} = \{ v \in V \mid q(u, \ldots, u, v) =0\}.$$ Thus ${\rm Ker} \phi^0_x = T_u \h{X} \otimes \h{x}^*$. It follows that $0 \in \bA^\circ$ and $E^0$ is isomorphic to $T_X$. \end{proof}

\begin{lemma}  \label{l.faithful} In the notation of Lemma \ref{l.circ},
 there is a Zariski open subset $\bA^{\rm faithful} \subset \bA^\circ$ with $0 \in \bA^{\rm faithful}$ such that   for any $\alpha,\alpha' \in \bA^{\rm faithful},$ the two vector bundles $E^{\alpha  }$ and $E^{\alpha'}$ on $X$ are  isomorphic only when $\alpha = \alpha'$.
\end{lemma}
\begin{proof}
For $E^{\alpha  }(-d)  := E^{\alpha  } \otimes_X \cO_X(-d)$,    define
\[
\bA^{\rm faithful} := \{ \alpha \in \bA^\circ \: : \: h^1(X,  E^{\alpha  }(-d) ) =1 \}.
\]
 Since $h^1(X, T_X(-d))=1$ from  Lemma \ref{basic} (iv), we see that $0 \in \bA^{\rm faithful}.$  Hence 
 \[
U := \{ \alpha \in \bA^\circ \: : \: h^1(X,  E^{\alpha  }(-d) )  \le 1 \}
 \]
  is a nonempty Zariski open subset  of $\bA^\circ$ by semicontinuity.
We further claim that $h^1(X,  E^{\alpha  }(-d) ) \neq 0$ for any $\alpha \in\bA^\circ$, which shows that $\bA^{\rm faithful}$ coincides with the open subset $U$.
For each $\alpha   \in \bA^\circ$,  there is an exact sequence of vector bundles
$$
0 \to E^{\alpha  } \to T_{\pp V} |_X \to \cO_X(d) \to 0 $$
from the last column of the diagram in Definition \ref{d.phi} (iv).   If $H^1(X,  E^{\alpha  }(-d) ) =0, $ this sequence splits.  This contradicts the simplicity of $T_{\pp^n} |_X$ proved in Lemma \ref{simple} (i).

Now suppose that $E^{\alpha  }$ and $E^{\alpha'}$ are isomorphic for some $\alpha,\alpha' \in \bA^{\rm faithful}$.   The condition $h^1(X,  E^{\alpha  }(-d) ) =1 =h^1(X, E^{\alpha'}(-d))$ says that   $T_{\pp^n} |_X$ is the unique non-split extension of $\cO_X(d)$ by $E^{\alpha  }$ (and also by $E^{\alpha'}$) up to isomorphisms.   Thus there should be  an isomorphism $\mu :  T_{\pp^n} |_X  \to T_{\pp^n} |_X $ fitting into the following commutative diagram:
 \[
\begin{tikzcd}
0 \arrow[r] & E^\alpha \arrow[r] \arrow[d, phantom, sloped, "\cong"]  &T_{\pp^n} |_X  \arrow[r]  \arrow[  d,   "\mu"]   & \cO_X(d) \arrow[r] \arrow[d, phantom, sloped, "="] &0 \\
0 \arrow[r] & E^{\alpha'} \arrow[r]  &T_{\pp^n} |_X  \arrow[r]   & \cO_X(d) \arrow[r] & 0
\end{tikzcd}
\]
Since  $T_{\pp^n} |_X$ is simple by Lemma \ref{simple} (i), the isomorphism  $\mu$ is a homothety.  Note that two maps $T_{\pp^n} |_X  \to \cO_X(d)$  in this diagram come from the homomorphisms $\phi^{\alpha}$ and $\phi^{\alpha'}$ in Definition \ref{d.phi} (ii),   modulo   the trivial subsheaf $\cO_X$.   This implies that $q+\alpha = c (q+\alpha')$ for some $c \in \cc$. From $\bA \cap \Sym^d V^*  = \{ 0\}$ in Lemma \ref{l.dim}, we conclude $c = 1$ and $\alpha = \alpha'$.
\end{proof}

\begin{proof}[Proof of Theorem \ref{t.Main}]:  We have constructed a family of sheaves $\mathbb{E} $ over $X$  parameterized by an affine space $\bA$ whose central fiber $\mathbb{E} |_{X \times \{0\}}$ is isomorphic to the tangent bundle $T_X$. By taking $0 \in \bA$ as the base point $o \in \bA$,  we see that
(i) - (iv) have been checked in  Lemmata \ref{l.dim},  \ref{l.circ} and  \ref{l.faithful}. (v) is immediate because    $E^0=T_X$ is stable. \end{proof}

\medskip
{\bf Acknowledgment} Jun-Muk Hwang would like to thank Mihai Paun and Thomas Peterell   for amusing discussions to locate the flaw in \cite{Hw}, which had motivated the current work.

\end{document}